\documentclass[11pt]{amsart}%
\usepackage{amssymb,amsmath,amsfonts,latexsym,amsthm,geometry,graphicx}
\usepackage[normalem]{ulem}
\usepackage{amsmath}
\usepackage{amsfonts}
\usepackage{color, soul}
\usepackage{amssymb}
\usepackage{tikz-cd}
\usepackage[all]{xy}%
\providecommand{\U}[1]{\protect\rule{.1in}{.1in}}
\theoremstyle{plain}
\newtheorem{theorem}{Theorem}[section]

\newtheorem{corollary}[theorem]{Corollary}

\theoremstyle{definition}
\newtheorem{definition}[theorem]{Definition}
\newtheorem{remark}[theorem]{Remark}

\numberwithin{equation}{section}

\begin{document}
	\title{A new criterion for the normalized Haar measure to be a Pietsch measure}
	\author[Alexandre Bispo]{Alexandre Bispo}
	\address{Departamento de Matem\'{a}tica \\
		Universidade Federal da Para\'{\i}ba \\
		58.051-900 - Jo\~{a}o Pessoa, Brazil.}
	\email{alexandre.cesar@academico.ufpb.br}
	\author[Renato Macedo]{Renato Macedo}
	\address{Centro de Ciências Exatas e Sociais Aplicadas \\
		Universidade Estadual da Para\'{\i}ba \\
		58.706-550 - Patos, Brazil.}
	\email{renato.burity@servidor.uepb.edu.br}
	\author[Joedson Santos]{Joedson Santos}
	\address{Departamento de Matem\'{a}tica \\
		Universidade Federal da Para\'{\i}ba \\
		58.051-900 - Jo\~{a}o Pessoa, Brazil.}
	\email{joedson.santos@academico.ufpb.br}

	\thanks{Alexandre Bispo is supported by Grant 2073/2023
		Paraíba State Research Foundation (FAPESQ)} 
	%\thanks{\textcolor{red}{Renato Macedo is partially supported by Capes}}
	\thanks{Joedson Santos is supported by CNPq Grant 406457/2023-9, Grant 403964/2024-5 and Grant 305655/2025-6} 
	\maketitle
	
	\begin{abstract}
		In this paper we present a new criterion to determine when the normalized Haar measure on a compact topological group is a Pietsch measure for nonlinear summing mappings. As a consequence, we provide a partial answer to a problem raised by Botelho et al. in \cite{haar botelho} motivated by a question posed to the authors by J. Diestel. It is explicitly shown that this criterion encompasses recent and new results as particular cases.

	\end{abstract}
	
	\section{Introduction}
	
	One of the central results in the theory of absolutely summing operators is the celebrated Pietsch Domination Theorem, established by A. Pietsch in 1967 in his classical paper \cite{stu} (for more details see \cite[Theorem 2.12]{diestel}). This theorem provides a fundamental characterization of absolutely $p$-summing operators in terms of domination by an integral with respect to a probability measure. More precisely, a continuous linear operator between Banach spaces $u:E\to F$ is absolutely $p$-summing, $1\le p<\infty$, if and only if there exist a constant $C>0$ and a regular Borel probability measure $\nu$ on the closed unit ball of the dual of $E$ with the weak$^{*}$ topology, $(B_{E^{*}}\sigma(E^{*},E))$, such that
	\begin{equation}\label{op.ab.s}
		\|u(x)\|_{F} \leq C \left(\int\limits_{B_{E^{*}}} |\alpha(x)|^{p} \, d\nu(\alpha) \right)^{1/p},
	\end{equation}
	for all $x \in E$. Any probability measure satisfying \eqref{op.ab.s} is called a \emph{Pietsch measure} for $u$.
	
	A particularly important situation arises when $E=C(K)$, the Banach space of scalar-valued continuous functions on a compact Hausdorff space $K$. In this case, the Pietsch Domination Theorem asserts that a linear operator $u:C(K)\to F$ is absolutely $p$-summing if and only if there exist a constant $C>0$ and a regular Borel probability measure $\nu$ on $K$ such that
	\begin{equation}\label{op.ab.s1}
		\|u(T)\|_{F} \leq C \left(\int\limits_{K} |T(\varphi)|^{p} \, d\nu(\varphi) \right)^{1/p},
	\end{equation}
	for all $T \in C(K)$. In this framework, the measure $\nu$ is again referred to as a Pietsch measure for $u$.
	
	Despite the structural importance of Pietsch measures, relatively little is known in general about their explicit form. However, additional information becomes available when the underlying compact space carries a group structure. Indeed, when $K=G$ is a compact Hausdorff topological group, $H$ is a closed translation invariant subspace of $C(G)$, $F$ is a Banach space, and $u:H\to F$ is a translation invariant absolutely $p$-summing linear operator, it follows that the normalized Haar measure on $G$ is a Pietsch measure for $u$ (see \cite[p.~56]{diestel} and \cite[Theorem 4.2]{haar botelho}).

	Recall that a Haar measure on a compact topological group $G$ is a nonzero Radon measure $\mu$ satisfying the translation invariance property
	\[
	\mu(\psi B)=\mu(B)
	\]
	for every $\psi\in G$ and every Borel set $B\subseteq G$. A classical result due to A. Haar \cite{A.Haar} ensures that there exists a unique Haar measure satisfying the normalization condition $\mu(G)=1$. This measure will be denoted throughout the paper by $\sigma_G$ and is commonly referred to as the \emph{normalized Haar measure} on $G$ (see also \cite{ca,folland,Neu,weil}).

	The Pietsch Domination Theorem has had a profound impact on Banach space theory, serving as a bridge between the theory of summing operators and measure theory. Because of its fundamental role, several extensions and abstract versions have been developed over the years. In particular, an abstract Pietsch Domination Theorem was obtained in \cite{adv}, providing a unified framework that encompasses a wide variety of domination results for classes of mappings extending the ideal of absolutely $p$-summing operators (see also \cite{unified PDT,general PDT}). 
	
	Motivated by this abstract approach, J. Diestel asked whether the normalized Haar measure on a compact topological group could still play the role of a Pietsch measure in such a general setting (for more abstract approaches, see \cite{BC, BMPS,LSS,israel,tunisia}). This question was answered affirmatively in \cite[Theorem 3.1]{haar botelho}, where it was shown that the normalized Haar measure indeed serves as a Pietsch measure in several abstract situations. As a consequence, the result applies to a number of important classes of summing operators, both linear and nonlinear, defined on closed translation invariant subspaces of $C(G)$. Nevertheless, certain relevant cases remain unresolved. In particular, as pointed out in \cite[Open Problem]{haar botelho}, the argument developed there does not seem to apply to the class of Lipschitz $p$-summing mappings. This leads naturally to the following open problem.
	
	\medskip
	
	\noindent \textbf{Open Problem.} \textit{Let $G$ be a compact Hausdorff topological group, $H$ be a closed translation invariant subspace of $C(G)$, $Y$ be a metric space and $u:H \to Y$ be a translation invariant Lipschitz $p$-summing mapping. Is the normalized Haar measure $\sigma_G$ a Pietsch measure for $u$?}
	
	\medskip
	
	The main goal of this paper is to provide a new criterion ensuring that the normalized Haar measure is a Pietsch measure in a broad class of situations. Our approach refines previous techniques and yields conditions that are applicable in settings where earlier results do not apply directly.
	
	The article is organized as follows. In Section \ref{sec notations} we recall the notation and preliminary results that will be used throughout the paper. In Section \ref{sec main result} we introduce the notion of a $D$-dominating measure on a compact Hausdorff topological group $G$, where $D:G\to [0,\infty)$ is a continuous function. This notion plays a central role in our approach and provides a useful criterion to determine when the normalized Haar measure is a Pietsch measure for several classes of (linear and nonlinear) summing operators. In Section \ref{OP} we present a partial solution to the open problem stated above. Finally, in Section 5 we introduce the class of Lipschitz $p$-semi-integral multilinear operators, establish a Pietsch-type domination theorem for this class, and prove that the normalized Haar measure is a Pietsch measure in this setting. We also show that our main result recovers, as a particular case, a result obtained in \cite[Corollary 5.2]{mastylo}, where it was proved that the normalized Haar measure is a Pietsch measure for the class of multilinear $\Phi$-summing operators, with $\Phi:[0,\infty)\to[0,\infty)$ being an Orlicz function.
	
	\section{Notations and background}\label{sec notations}
	
	To describe our aims, we introduce the notation used in this paper. Let $K$ be a compact Hausdorff topological space and let $C(K)$ be the Banach space of all continuous mappings from $K$ into $\mathbb{K}$ ($\mathbb{K} = \mathbb{R}$ or $\mathbb{C}$), equipped with the sup norm given by
	\begin{equation*}
		\|f\|_{C(K)} := \sup\limits_{\varphi \in K} |f(\varphi)|.
	\end{equation*}
	When $K = G$ is a compact Hausdorff topological group we use the usual multiplicative notation for the operation on $G$. For any $\psi \in G$, we define the left translation homeomorphism $I_{\psi}: G \to G$ by
	\begin{equation*}
		I_{\psi}(\varphi):= \psi\varphi, \hbox{for all $\varphi \in G$}.
	\end{equation*}
	Note that $I_{\psi} \circ I_{\xi} = I_{\psi\xi}$, for all $\psi,\xi \in G$. Hence, denoting by $e \in G$ its neutral element, one has $I_{e} = Id_{G}$ and consequently $I_{\psi}^{-1} = I_{\psi^{-1}}$, for all $\psi \in G$.
	
	%\begin{remark}
	Let $m \in \mathbb{N}$, $G_{1},\dots,G_{m}$ be compact Hausdorff topological groups and $e_{i}$ be the neutral element of $G_{i}$, for every $i \in \{1,\dots,m\}$. Consider $G:= G_{1} \times \dots \times G_{m}$ the compact Hausdorff topological group with neutral element $e=(e_{1},\dots,e_{m})$, with the usual coordinatewise operation and equipped with the product topology. Given $\psi=(\psi_{1},\dots,\psi_{m}) \in G$, the translation homeomorphism $I_{\psi}: G \to G$ is defined by
	\begin{equation*}
		I_{\psi}(\varphi):= (I_{\psi_{1}}(\varphi_{1}),\dots,I_{\psi_{m}}(\varphi_{m})) = (\psi_{1}\varphi_{1},\dots,\psi_{m}\varphi_{m}), \hbox{ for all $\varphi=(\varphi_{1},\dots\varphi_{m}) \in G$}.
	\end{equation*}
	%\end{remark}
	
	\begin{definition}\label{def translation invariant}
		Let $m \in \mathbb{N}$, $G_{1},\dots,G_{m}, G$ be compact Hausdorff topological groups and $A$ be an arbitrary set.
		\begin{itemize}
			\item[(i)] A non-void subset $H$ of $C(G)$ is called \textit{translation invariant} if for all $\psi \in G$ and all $T \in H$, we have $T \circ I_{\psi} \in H$.
			
			\item[(ii)] Let $H$ be a translation invariant subset of $C(G)$. A map $u: H \to A$ is called \textit{translation invariant} if, for all $\psi \in G$ and all $T \in H$, we have $u(T \circ I_{\psi}) = u(T)$.
			
			\item[(iii)] Let $H_{i}$ be a translation invariant subset of $C(G_{i})$, with $i \in \{1,\dots,m\}$. A map $u: H_{1} \times \dots \times H_{m} \to A$ is called \textit{translation $m$-invariant} if, for all $(T_{1},\dots,T_{m}) \in H_{1} \times \dots \times H_{m}$ and all $(\psi_{1},\dots,\psi_{m}) \in G_{1} \times \dots \times G_{m}$, we have $	u(T_{1} \circ I_{\psi_{1}},\dots,T_{m} \circ I_{\psi_{m}}) = u(T_{1},\dots,T_{m})$.

		\end{itemize}
	\end{definition}

	Now, let us list some facts on measure theory that will be useful throughout this paper.
	
	\begin{remark}$($see \cite[p. 304]{bogachev v2} and \cite[p. 33]{weil} $)$\label{remark haar}
		Let $G$ be a compact Hausdorff topological group and $\mu$ be a Haar measure on $G$. Then, for all non-negative measurable function $f: G \to \mathbb{R}$ and for all $\varphi \in G$, we have
		\begin{equation*}
			\int\limits_{G} f(\varphi\psi) \, d\mu(\psi) = \int\limits_{G} f(\psi) \, d\mu(\psi).
		\end{equation*}		
	\end{remark}
	
	\begin{definition}
		Let $(M,\mu)$ be a measure space, $N$ be a measurable space and $I: M \to N$ be a measurable map. The \textit{image measure} of $\mu$ via $I$ is the measure on $N$, denoted by $I\mu$, given by
		\begin{equation*}
			I\mu(B):= \mu(I^{-1}(B)), \hbox{ for all measurable subset $B \subseteq N$}.
		\end{equation*}
	\end{definition}
	
	\begin{remark}(see \cite[Theorem 3.6.1]{bogachev v1} and \cite[Theorem 9.1.1.(i)]{bogachev v2})\label{remark image measure}
		Let $(M,\mu)$ be a measure space, $N$ be a measurable space and $I:M \to N$ be a measurable map.
		\begin{itemize}
			\item[(i)] If $g: N \to \mathbb{R}$ is a non-negative function such that $g \circ I: M \to \mathbb{R}$ is an integrable function on $M$, then $g$ is integrable on $N$ and
			\begin{equation*}
				\int\limits_{N} g(y) \, d(I\mu)(y) = \int\limits_{M} g \circ I(x) \, d\mu(x).
			\end{equation*}
			
			\item[(ii)] If $M$ and $N$ are Hausdorff spaces, $I$ is a continuous map and $\mu$ is a Radon measure, then $I\mu$ is also a Radon measure.  
			
			\item[(iii)] If $M=N=G$ is a compact Hausdorff topological group, $\mu$ is a Radon measure and $I=I_{\varphi}$ is the left translation homeomorphism, then $\mu$ is a Haar measure if and only if $(I_{\varphi})\mu = \mu$ for all $\varphi \in G$.
		\end{itemize}
	\end{remark}

	\section{Main Results}\label{sec main result}
	
	In this section we present our main result that will be used to provide a positive partial answer to the Open Problem stated in the introduction. The key point is to show that the normalized Haar measure on $G$ satisfies a suitable integral domination property.

	\begin{definition}
		Let $G$ be a compact Hausdorff topological group, $D: G \to [0,\infty)$ be a continuous function and $\nu$ be a regular Borel probability measure on $G$. We say that $\nu$ is \textit{$D$-dominating} if 
		\begin{equation*}
			1 \leq \int\limits_{G} D(\psi\varphi) \, d\nu(\varphi),\ \hbox{ for all $\psi \in G$.}.
		\end{equation*}
	\end{definition}
	
	\begin{theorem}\label{main theorem}
		Let $G$ be a compact Hausdorff topological group and $D: G \to [0,\infty)$ be a continuous function. If there is a regular Borel probability measure $\nu$ on $G$ such that $\nu$ is $D$-dominating, then $\sigma_{G}$ is $D$-dominating.
	\end{theorem}
	
	\begin{proof}
		Let $I: G \times G \to G$ be the continuous map defined by
		\begin{equation*}
			I(\psi,\varphi):= I_{\psi}(\varphi) = \psi\varphi, \hbox{ for all $(\psi,\varphi) \in G \times G$}.
		\end{equation*}
		Since $\nu$ is $D$-dominating, then for all $\psi \in G$, we have
		\begin{equation}\label{proof sub}
			1 \leq \int\limits_{G} D(\psi\varphi) \, d\nu(\varphi) = \int\limits_{G} D \circ I(\psi,\varphi) \, d\nu(\varphi).
		\end{equation}
		As the map $D \circ I: G \times G \to [0,\infty)$ is continuous, it follows that $D \circ I$ is integrable. Hence, by Fubini's Theorem, the function
		\begin{equation*}
			\psi \in G \mapsto \int\limits_{G} D \circ I(\psi,\varphi) \, d\nu(\varphi)
		\end{equation*}
		is Borel measurable and
		\begin{equation}\label{proof fubini}
			\int\limits_{G} \int\limits_{G} D \circ I(\psi,\varphi) \, d\nu(\varphi) \, d\sigma_{G}(\psi) = \int\limits_{G \times G} D \circ I(\psi,\varphi) \, d(\sigma_{G} \times \nu)(\psi,\varphi).
		\end{equation}
		Considering $\mu:= I(\sigma_{G} \times \nu)$ the image measure on $G$ of $\sigma_{G} \times \nu$ via $I$, it follows from \eqref{proof sub}, \eqref{proof fubini} and Remark \ref{remark image measure}(i) that
		\begin{equation}\label{blabla}
			1 \leq \int\limits_{G \times G} D \circ I(\psi,\varphi) \, d(\sigma_{G} \times \nu)(\psi,\varphi) = \int\limits_{G} D(\varphi) \, dI(\sigma_{G} \times \nu)(\varphi) = \int\limits_{G} D(\varphi) \, d\mu(\varphi).
		\end{equation}		
		
		Now we will show that $\mu = \sigma_{G}$. Indeed, by Remark \ref{remark image measure}(ii), $\mu$ is a Radon measure. Furthermore, for every Borel subset $B \subseteq G$ and every $\psi \in G$, we get
		\begin{equation}\label{proof mu}
			I^{-1}(B)_{\psi} := \{\varphi \in G: (\psi,\varphi) \in I^{-1}(B)\} = \{\varphi \in G: I_{\psi}(\varphi) = I(\psi,\varphi) \in B\} = (I_{\psi})^{-1}(B).
		\end{equation}
		Using the product measure, we obtain
		\begin{equation}\label{proof mu(B)}
			\begin{aligned}
				\mu(B) & = I(\sigma_{G} \times \nu)(B) = (\sigma_{G} \times \nu)(I^{-1}(B)) \\
				& = \int\limits_{G} \nu(I^{-1}(B)_{\psi}) \, d\sigma_{G}(\psi) = \int\limits_{G} \nu((I_{\psi})^{-1}(B))) \, d\sigma_{G}(\psi).
			\end{aligned}
		\end{equation}
		For every $\varphi \in G$, by \eqref{proof mu} and \eqref{proof mu(B)}, we have
		\begin{equation}\label{proof conta 1}
			\begin{aligned}
				(I_{\varphi})\mu(B) & = \mu((I_{\varphi})^{-1}(B)) = \int\limits_{G} \nu((I_{\psi})^{-1}((I_{\varphi})^{-1}(B)))) \, d\sigma_{G}(\psi) \\
				& = \int\limits_{G} \nu(I_{\psi^{-1}}(I_{\varphi^{-1}}(B))) \, d\sigma_{G}(\psi) = \int\limits_{G} \nu(I_{\psi^{-1}\varphi^{-1}}(B)) \, d\sigma_{G}(\psi) \\
				& = \int\limits_{G} \nu(I_{(\varphi\psi)^{-1}}(B)) \, d\sigma_{G}(\psi) = \int\limits_{G} \nu((I_{\varphi\psi})^{-1}(B)) \, d\sigma_{G}(\psi) \\
				& = \int\limits_{G} \nu(I^{-1}(B)_{\varphi\psi}) \, d\sigma_{G}(\psi).
			\end{aligned}
		\end{equation}
		Of course the function $\psi \in G \mapsto \nu(I^{-1}(B)_{\psi})$ is non-negative and measurable on $G$, thus the same is true for the function $\psi \in G \mapsto \nu(I^{-1}(B)_{\varphi\psi})$. Then, by Remark \ref{remark haar}, \eqref{proof mu}, \eqref{proof mu(B)} and \eqref{proof conta 1}, we have
		\begin{equation*}\label{proof conta 2}
			(I_{\varphi})\mu(B)=\int\limits_{G} \nu(I^{-1}(B)_{\varphi\psi}) \, d\sigma_{G}(\psi) = \int\limits_{G} \nu(I^{-1}(B)_{\psi}) \, d\sigma_{G}(\psi) = \int\limits_{G} \nu((I_{\psi})^{-1}(B)) \, d\sigma_{G}(\psi) = \mu(B).
		\end{equation*}
		Thus, by Remark \ref{remark image measure}(iii) it follows that $\mu$ is a Haar measure. Moreover
		\begin{equation*}
			\mu(G) = I(\sigma_{G} \times \nu)(G) = (\sigma_{G} \times \nu)(I^{-1}(G)) = (\sigma_{G} \times \nu)(G \times G) = \sigma_{G}(G)\nu(G) = 1.
		\end{equation*}
		Therefore $\mu = \sigma_{G}$. Finally, we can use (\ref{blabla}) and Remark \ref{remark haar} to obtain
		\begin{equation*}
			1 \leq \int\limits_{G} D(\varphi) \, d\mu(\varphi) = \int\limits_{G} D(\varphi) \, d\sigma_{G}(\varphi) = \int\limits_{G} D(\psi\varphi) \, d\sigma_{G}(\varphi),
		\end{equation*}
		what concludes the proof.
	\end{proof}

	\section{A partial answer to Open Problem}\label{OP}
	
	Let $X$ be a pointed metric space with base point $0_{X} \in X$. We denote by $X^{\#}$ its \textit{Lipschitz dual space}, that is, $X^{\#}$ is the Banach space of all Lipschitz functions $f: X \to \mathbb{K}$ such that $f(0_{X}) = 0$, endowed with the norm
	\begin{equation*}
		\|f\|_{X^{\#}}:= \sup\limits_{^{x,q \in X}_{x \neq q}} \dfrac{|f(x)-f(q)|}{d_{X}(x,q)}.
	\end{equation*}

	Inspired by the remarkable work \cite{Far.Joh.}, Mastylo and Pérez \cite{mastyloperez} introduced an interesting variant of nonlinear Lipschitz summing maps between metric spaces. Let $W\subseteq B_{X^{\#}}$ be a compact subset with respect to the topology of the pointwise convergence on $X$ and let $1 \leq p < \infty$. A Lipschitz mapping $u: X \to Y$ from a pointed metric space $X$ to a metric space $Y$ is called \textit{Lipschitz $p$-$W$-summing} if there is a constant $C>0$ such that, for all $n \in \mathbb{N}$ and $x_{1},\dots,x_{n},q_{1},\dots,q_{n} \in X$, we have
	\begin{equation}\label{lipschitz}
		\left(\sum\limits_{j=1}^{n} d_{Y}(u(x_{j}),u(q_{j}))^{p}\right)^{1/p} \leq C \sup\limits_{f\in W} \left(\sum\limits_{j=1}^{n} |f(x_{j})-f(q_{j})|^{p} \right)^{1/p}.
	\end{equation}
	The infimum of all constants that satisfies \eqref{lipschitz} is denoted by $\pi_{p}^{W}(u)$ and the set of all Lipschitz $p$-$W$-summing mappings from $X$ to $Y$ is denoted by $\Pi_{p}^{W}(X;Y)$ (in particular, when $X$ is a Banach space and $W=B_{X^{*}}$ we use the notation $\pi_{p}^{*}(u)$ and $\Pi_{p}^{*}(X;Y)$). This definition unifies some concepts of classes of summing operators. Let $X, Y$ be metric spaces and $E, F$ be Banach spaces. Then
	
	\begin{itemize}
		\item  a linear operator $u : E \to F$ is Lipschitz $p$-$B_{E^{*}}$-summing if and only if it is $p$-summing in the classical sense.
		
		\item  a Lipschitz operator $u : E \to F$ is Lipschitz $p$-$B_{E^{*}}$-summing if and only if it is Lipschitz $p$-dominated (see \cite[Theorem 3.2]{ChenZeheng}).
		
		\item  a Lipschitz operator $u : X \to Y$ is Lipschitz $p$-$B_{X^{\#}}$-summing if and only if it is Lipschitz $p$-summing (see \cite{Far.Joh.}).
	\end{itemize}
	
	We now focus on the case where $u: E \to Y$ is a Lipschitz operator and $W=B_{E^{*}}$. Our first result proves that these operators satisfy a Pietsch-type domination theorem.

	\begin{theorem} For a Lipschitz mapping $u: E \to Y$, the following statements are equivalent:
		
		(i) $u\in \Pi_{p}^{*}(E;Y)$;
		
		(ii) there is a constant $C>0$ such that, for all $n \in \mathbb{N}$ and $x_{1},\dots,x_{n},q_{1},\dots,q_{n} \in E$, we have
		\begin{equation*}\label{lipschitz dominated}
			\left(\sum\limits_{j=1}^{n} d_{Y}(u(x_{j}),u(q_{j}))^{p}\right)^{1/p} \leq C \sup\limits_{\alpha \in B_{E^{*}}} \left(\sum\limits_{j=1}^{n} |\alpha(x_{j} - q_{j})|^{p} \right)^{1/p};
		\end{equation*}
		
		(iii) there are a constant $C>0$ and a regular Borel probability measure $\nu$ on $B_{E^{*}}$ such that, for all $x,q \in E$, we have
		\begin{equation*}\label{lipschitz domination 2}
			d_{Y}(u(x),u(q)) \leq C \left(\int\limits_{B_{E^{*}}} |\alpha(x-q)|^{p} \, d\nu(\varphi) \right)^{1/p}.
		\end{equation*}
		In any case $\pi_{p}^{*}(u)=\inf\{{C:\ C\ \text{as in}\ (ii)}\}=\inf\{{C:\ C\ \text{as in}\ (iii)}\}$.
		
	\end{theorem}
		
		\begin{proof}
			It follows directly from the definition that $(i)\Leftrightarrow(ii)$. And $(ii)\Leftrightarrow(iii)$  is a consequence of the abstract versions of Pietsch Domination Theorem \cite{unified PDT, general PDT}.
		\end{proof}

	It will be very convenient to have at our disposal a version of this result for the case where $E=H$ is a closed subspace of $C(K)$, for some compact Hausdorff topological space $K$.
	
	\begin{corollary} \label{coro importante}Let $K$ be a compact Hausdorff topological space and $H$ be a closed subspace of $C(K)$. For a Lipschitz mapping $u: H \to Y$ the following statements are equivalent:
		
		$(i)$ $u\in \Pi_{p}^{*}(H;Y)$;
		
		$(ii)$' for all $n \in \mathbb{N}$ and $T_{1},\dots,T_{n},S_{1},\dots,S_{n} \in H$, we have
		\begin{equation*}
			\left(\sum\limits_{j=1}^{n}  d_{Y}(u(T_{j}),u(S_{j}))^{p} \right)^{1/p} \leq \pi_{p}^{*}(u) \sup\limits_{\varphi \in K} \left(\sum\limits_{j=1}^{n} |T_{j}(\varphi) - S_{j}(\varphi)|^{p} \right)^{1/p};
		\end{equation*}
		
		$(iii)$' there is a regular Borel probability measure $\nu$ on $K$ such that, for all $T,S \in H$, we have
		\begin{equation*}
			d_{Y}(u(T),u(S)) \leq \pi_{p}^{*}(u) \left(\int\limits_{K} |T(\varphi)-S(\varphi)|^{p} \, d\nu(\varphi) \right)^{1/p}.
		\end{equation*}
	\end{corollary}
	
	To derive $(ii)$' from $(ii)$ it is enough to use that for all $n \in \mathbb{N}$ and $T_{1},\dots,T_{n} \in H$, we have
	\begin{equation*}\label{igualdade matadora}
		\sup\limits_{\alpha \in B_{H^{*}}} \sum\limits_{j=1}^{n} |\alpha(T_{j})-\alpha(S_{j})|^{p} = \sup\limits_{\beta \in B_{C(K)^{*}}} \sum\limits_{j=1}^{n} |\beta(T_{j})-\beta(S_{j})|^{p} = \sup\limits_{\varphi \in K} \sum\limits_{j=1}^{n} |T_{j}(\varphi)-S_{j}(\varphi)|^{p}
	\end{equation*}
	(see \cite[p. 280]{haar botelho}). Therefore, we can use again the abstract versions of the Pietsch Domination Theorem \cite{unified PDT, general PDT} to obtain $(iii)$'.
	
	Now we shall prove that our main result works as a criterion to determine when the normalized Haar measure is a Pietsch measure for Lipschitz $p$-$B_{H^{*}}$-summing operators, and this result can not be straightforwardly obtained from \cite[Theorem 3.1]{haar botelho}.
	
	\begin{theorem}\label{caso lip}
		Let $G$ be a compact Hausdorff topological group, $H$ be a closed translation invariant subspace of $C(G)$, $Y$ be a metric space, $1 \leq p < \infty$ and $u: H \to Y$ be a Lipschitz $p$-$B_{H^{*}}$-summing translation invariant map. Then the normalized Haar measure on $G$ is a Pietsch measure for $u$.
	\end{theorem} 
	
	\begin{proof}
		We fix any $(T,S) \in H \times H$ such that $d_{Y}(u(T),u(S))^{p} \neq 0$ and define $D: G \to [0,\infty)$ by
		\begin{equation*}
			D(\varphi) := \dfrac{\pi_{p}^{*}(u)^{p}|T(\varphi)-S(\varphi)|^{p}}{d_{Y}(u(T),u(S))^{p}}. 
		\end{equation*}
		Clearly $D$ is well defined and is continuous. Since $H$ is a translation invariant subset of $C(G)$ and $u$ is a translation invariant map, it follows that
		\begin{equation*}
			d_{Y}(u(T\circ I_{\psi}),u(S\circ I_{\psi})) = d_{Y}(u(T),u(S)) \neq 0,
		\end{equation*}
		for all $\psi \in G$. Since $u$ is Lipschitz $p$-$B_{H^{*}}$-summing, there is a regular Borel probability measure $\nu$ on $G$ such that
		\begin{equation*}
			d_{Y}(u(T\circ I_{\psi}),u(S\circ I_{\psi}))^{p} \leq \pi_{p}^{*}(u)^{p} \int\limits_{G} |T\circ I_{\psi}(\varphi)-S\circ I_{\psi}(\varphi)|^{p} \, d\nu(\varphi),
		\end{equation*}
		for all $\psi \in G$. Thus
	
		\begin{equation*}
			\begin{aligned}
				1 & \leq \int\limits_{G} \dfrac{\pi_{p}^{*}(u)^{p} |T\circ I_{\psi}(\varphi)-S\circ I_{\psi}(\varphi)|^{p}}{d_{Y}(u(T\circ I_{\psi}),u(S\circ I_{\psi}))^{p}} \, d\nu(\varphi)  \\
				& = \int\limits_{G} \dfrac{\pi_{p}^{*}(u)^{p}|T(\psi\varphi)-S(\psi\varphi)|^{p}}{d_{Y}(u(T),u(S))^{p}} \, d\nu(\varphi) = \int\limits_{G} D(\psi\varphi) \, d\nu(\varphi), 
			\end{aligned}
		\end{equation*}
		for all $\psi \in G$. Hence, $\nu$ is $D$-dominating. By Theorem \ref{main theorem}, $\sigma_{G}$ is also $D$-dominating. Then, for all $\psi \in G$, we have
		\begin{equation*}
			1 \leq \int\limits_{G} D(\psi\varphi) \, d\sigma_{G}(\varphi) = \int\limits_{G} D(\varphi) \, d\sigma_{G}(\varphi) = \int\limits_{G} \dfrac{\pi_{p}^{*}(u)^{p}|T(\varphi)-S(\varphi)|^{p}}{d_{Y}(u(T),u(S))^{p}} \, d\sigma_{G}(\varphi),
		\end{equation*}
		that is,
		\begin{equation}\label{conta lipschitz}
			d_{Y}(u(T),u(S)) \leq \pi_{p}^{*}(u) \left(\int\limits_{G} |T(\varphi)-S(\varphi)|^{p} \, d\sigma_{G}(\varphi) \right)^{1/p}.
		\end{equation}
		Of course the inequality \eqref{conta lipschitz} is also satisfied for all $(T,S) \in H \times H$ such that $d_{Y}(u(T),u(S)) = 0$. Consequently, $\sigma_{G}$ is a Pietsch measure for $u$.
	\end{proof}

\begin{remark}
	It is well known that if a linear operator defined on a closed subspace $H$ of $C(K)$, where $K$ is a compact Hausdorff space, is absolutely $p$-summing, then there exists a regular Borel probability measure that provides an integral domination of the form \eqref{op.ab.s1}. However, as pointed out in \cite{Far.Joh.}, such an implication does not hold in the nonlinear setting of Lipschitz $p$-summing operators. It is also worth noting that the geometric structure of $X^{\#}$ is still poorly understood and, in most situations, difficult to handle. So the relevance of Corollary \ref{coro importante} lies in the fact that it provides an analogous domination property for the Lipschitz $p$-summing operators when we use $B_{H^{*}}$ instead of $B_{H^{\#}}$. Moreover, since every Lipschitz $p$-$B_{H^{*}}$-summing operator is Lipschitz $p$-summing, we conjecture that Theorem \ref{caso lip} provides the best possible answer to the open problem mentioned in \cite{haar botelho}.
\end{remark}
	
	\section{ Other applications}\label{MA}
	
	Following what was done in the previous section, Theorem \ref{main theorem} can be easily invoked in order to prove that normalized Haar measure on $G$ is also a Pietsch measure for all classes of summing operators present in \cite{haar botelho}. It is also interesting to notice that we can use Theorem \ref{main theorem} to encompass other cases that are not covered in the general approach from \cite[Theorem 3.1]{haar botelho}.
	
	\subsection{Lipschitz $p$-semi-integral multilinear operators} 
	
	We begin this subsection introducing the notion Lipschitz $p$-semi-integral operators. Let $m \in \mathbb{N}$, $E_{1},\dots,E_{m}$ and $F$ be Banach spaces and $1 \leq p < \infty$. A continuous $m$-linear operator $u: E_{1} \times \dots \times E_{m} \to F$ is \textit{Lipschitz $p$-semi-integral} if there is a constant $C>0$ such that, for all $n \in \mathbb{N}$ and $(x_{1,j},\dots,x_{m,j}),(q_{1,j},\dots,q_{m,j}) \in E_{1} \times \dots \times E_{m}$, $j \in \{1,\dots,n\}$, we have
	\begin{equation}\label{lipschitz semi}
		\begin{aligned}
			& \left(\sum\limits_{j=1}^{n}\|u(x_{1,j},\dots,x_{m,j}) - u(q_{1,j},\dots,q_{m,j})\|_{F}^{p}\right)^{1/p}\leq \\
			& \leq C \sup\limits_{(\alpha_{1},\dots,\alpha_{m}) \in B_{E_{1}^{*}} \times \dots \times B_{E_{m}^{*}}} \left(\sum\limits_{j=1}^{n}|\alpha_{1}(x_{1,j})\dots\alpha_{m}(x_{m,j}) - \alpha_{1}(q_{1,j})\dots\alpha_{m}(q_{m,j})|^{p}\right)^{1/p}.
		\end{aligned}
	\end{equation}
	The infimum of all constants that satisfies \eqref{lipschitz semi} will be denoted by $\pi_{p,si}^{L}(u)$.
	
	When $K_{1},\dots,K_{m}$ are compact Hausdorff topological spaces and $H_{i}$ is a closed subspace of $C(K_{i})$, $i \in \{1,\dots,m\}$, then for all $n \in \mathbb{N}$ and $(T_{1,j},\dots,T_{m,j}),(S_{1,j},\dots,S_{m,j}) \in H_{1} \times \dots \times H_{m}$, $j \in \{1,\dots,n\}$, one has
	\begin{equation}\label{igualdade matadora multi lip}
		\begin{aligned} 
			&\sup\limits_{(\alpha_{1},\dots,\alpha_{m}) \in B_{H_{1}^{*}} \times \dots \times B_{H_{m}^{*}}} \sum\limits_{j=1}^{n}|\alpha_{1}(T_{1,j})\dots\alpha_{m}(T_{m,j}) - \alpha_{1}(S_{1,j})\dots\alpha_{m}(S_{m,j})|^{p} = \\
			&= \sup\limits_{(\varphi_{1},\dots,\varphi_{m}) \in K_{1} \times \dots \times K_{m}} \sum\limits_{j=1}^{n}|T_{1,j}(\varphi_{1})\dots T_{m,j}(\varphi_{m}) - S_{1,j}(\varphi_{1})\dots S_{m,j}(\varphi_{m})|^{p}.
		\end{aligned}
	\end{equation}
	The proof of this can be found in \cite[Theorem 2]{CD}.
	
	As a consequence of (\ref{igualdade matadora multi lip}) we obtain a version of the Pietsch Domination Theorem, whose proof follows from a simple application of \cite[Theorem 3.1]{general PDT} (see also \cite{unified PDT, adv}).
	
	\begin{theorem}\label{pdt semi}
		Let $m \in \mathbb{N}$, let $K_{1},\dots,K_{m}$ be compact Hausdorff topological spaces, $H_{i}$ be a closed subspace of $C(K_{i})$, $i \in \{1,\dots,m\}$, $F$ be a Banach space and $1 \leq p < \infty$. A continuous $m$-linear operator $u: H_{1} \times \dots \times H_{m} \to F$ is Lipschitz $p$-semi-integral if and only if there are a constant $C>0$ and a regular Borel probability measure $\nu$ on $K_{1} \times \dots \times K_{m}$ endowed with the product topology such that, for all $(T_{1},\dots,T_{m}),(S_{1},\dots,S_{m}) \in H_{1} \times \dots \times H_{m}$, we have
		\begin{equation}\label{domination semi}
			\begin{aligned}
				& \|u(T_{1},\dots,T_{m}) - u(S_{1},\dots,S_{m})\|_{F} \leq\\
				& \leq C \left(\int\limits_{K_{1} \times \dots \times K_{m}} |T_{1}(\varphi_{1})\dots T_{m}(\varphi_{m}) - S_{1}(\varphi_{1})\dots S_{m}(\varphi_{m})|^{p} \, d\nu(\varphi) \right)^{1/p}.
			\end{aligned}
		\end{equation}
	\end{theorem}
The infimum of all constants that satisfies \eqref{domination semi} is precisely $\pi_{p,si}^{L}(u)$. A measure $\nu$ on $G_{1} \times \dots \times G_{m}$ that satisfies \eqref{domination semi} is also called Pietsch measure for $u$.
	
	\begin{theorem}
		Let $m \in \mathbb{N}$, $G_{1},\dots,G_{m}$ be compact Hausdorff topological groups, $H_{i}$ be a closed translation invariant subspace of $C(G_{i})$, $i \in \{1,\dots,m\}$, $F$ be a Banach space,  $1 \leq p < \infty$ and $u: H_{1} \times \dots \times H_{m} \to F$ be a translation $m$-invariant Lipschitz $p$-semi-integral operator. Then the normalized Haar measure on $G_{1} \times \dots \times G_{m}$ is a Pietsch measure for $u$.
	\end{theorem}
	
	\begin{proof}
		Let $G = G_{1} \times \dots \times G_{m}$ and we fix $(T_{1},\dots,T_{m}),(S_{1},\dots,S_{m}) \in H_{1} \times \dots \times H_{m}$ such that $\|u(T_{1},\dots,T_{m}) - u(S_{1},\dots,S_{m})\|_{F} \neq 0$. Define $D: G \to [0,\infty)$ given by
		\begin{equation*}
			D(\varphi) = \dfrac{\pi_{p,si}^{L}(u)^{p}|T_{1}(\varphi_{1})\dots T_{m}(\varphi_{m}) - S_{1}(\varphi_{1})\dots S_{m}(\varphi_{m})|^{p}}{\|u(T_{1},\dots,T_{m}) - u(S_{1},\dots,S_{m})\|_{F}^{p}}, \hbox{ for all $\varphi = (\varphi_{1},\dots,\varphi_{m}) \in G$}.
		\end{equation*}
		Clearly $D$ is well defined and continuous. Since $H_{i}$ is a translation invariant subset of $C(G_{i})$, for each $i\in \{1,\dots,m\}$, and $u$ is a translation $m$-invariant, we obtain
		\begin{equation*}
			\|u(T_{1} \circ I_{\psi_{1}}, \dots, T_{m} \circ I_{\psi_{m}}) - u(S_{1} \circ I_{\psi_{1}}, \dots, S_{m} \circ I_{\psi_{m}})\|_{F} = \|u(T_{1},\dots,T_{m}) - u(S_{1},\dots,S_{m})\|_{F} \neq 0.
		\end{equation*}
		By Theorem \ref{pdt semi}, there exist a regular Borel probability measure $\nu$ on $G$ such that for all $\psi=(\psi_{1},\dots,\psi_{m}) \in G$, we have
		\begin{equation*}
			\begin{aligned}
				& \|u(T_{1} \circ I_{\psi_{1}},\dots,T_{m} \circ I_{\psi_{m}}) - u(S_{1} \circ I_{\psi_{1}},\dots,S_{m} \circ I_{\psi_{m}})\|_{F}^{p} \leq\\
				& \leq \pi_{p,si}^{L}(u)^{p} \int\limits_{G} |T_{1}\circ I_{\psi_{1}}(\varphi_{1})\dots T_{m}\circ I_{\psi_{m}}(\varphi_{m}) - S_{1}\circ I_{\psi_{1}}(\varphi_{1})\dots S_{m}\circ I_{\psi_{m}}(\varphi_{m})|^{p} \, d\nu(\varphi).
			\end{aligned}
		\end{equation*}
		It follows that
		\begin{equation*}
			\begin{aligned}
				1 & \leq \int\limits_{G} \dfrac{\pi_{p,si}^{L}(u)^{p} |T_{1}\circ I_{\psi_{1}}(\varphi_{1})\dots T_{m}\circ I_{\psi_{m}}(\varphi_{m}) - S_{1}\circ I_{\psi_{1}}(\varphi_{1})\dots S_{m}\circ I_{\psi_{m}}(\varphi_{m})|^{p}}{\|u(T_{1} \circ I_{\psi_{1}},\dots,T_{m} \circ I_{\psi_{m}}) - u(S_{1} \circ I_{\psi_{1}},\dots,S_{m} \circ I_{\psi_{m}})\|_{F}^{p}} \, d\nu(\varphi) \\
				& = \int\limits_{G} \dfrac{\pi_{p,si}^{L}(u)^{p}|T_{1}(\psi_{1}\varphi_{1})\dots T_{m}(\psi_{m}\varphi_{m}) - S_{1}(\psi_{1}\varphi_{1})\dots S_{m}(\psi_{m}\varphi_{m})|^{p}}{\|u(T_{1},\dots,T_{m}) - u(S_{1},\dots,S_{m})\|_{F}^{p}} \, d\nu(\varphi) = \int\limits_{G} D(\psi\varphi) \, d\nu(\varphi).
			\end{aligned}
		\end{equation*}
		Thus $\nu$ is $D$-dominating. By Theorem \ref{main theorem}, $\sigma_{G}$ is $D$-dominating. Hence, for all $\psi \in G$, we have
		\begin{equation*}
			\begin{aligned}
				1 & \leq \int\limits_{G} D(\psi\varphi) \, d\sigma_{G}(\varphi) = \int\limits_{G} D(\varphi) \, d\sigma_{G}(\varphi) \\
				& = \int\limits_{G} \dfrac{\pi_{p,si}^{L}(u)^{p}|T_{1}(\varphi_{1})\dots T_{m}(\varphi_{m}) - S_{1}(\varphi_{1})\dots S_{m}(\varphi_{m})|^{p}}{\|u(T_{1},\dots,T_{m}) - u(S_{1},\dots,S_{m})\|_{F}^{p}} \, d\sigma_{G}(\varphi),
			\end{aligned}
		\end{equation*}
		then
		\begin{equation*}\label{conta lipsemi}
			\begin{aligned}
				& \|u(T_{1},\dots,T_{m}) - u(S_{1},\dots,S_{m})\|_{F}\leq \\
				& \leq \pi_{p,si}(u)^{p} \left(\int\limits_{G} |T_{1}(\varphi_{1})\dots T_{m}(\varphi_{m}) - S_{1}(\varphi_{1})\dots S_{m}(\varphi_{m})|^{p} \, d\sigma_{G}(\varphi) \right)^{1/p}.
			\end{aligned}
		\end{equation*}
		Therefore we conclude that $\sigma_{G}$ is a Pietsch measure for $u$.
	\end{proof}

	\subsection{$\Phi$-summing multilinear operators, where $\Phi$ is an Orlicz function}
	Let $(\Omega,\Sigma,\mu)$ be a $\sigma$-finite measure space and $\Phi: [0,+\infty) \to [0,+\infty)$ be an \textit{Orlicz function} (that is, $\Phi$ is a convex, increasing and continuous function, with $\Phi(0) = 0$). The \textit{Orlicz space} $L_{\Phi}(\mu)$ on $(\Omega,\Sigma,\mu)$ is defined to be the space of all (equivalence classes of) functions $\mu$-measurable $f: \Omega \to \mathbb{R}$ such that
	\begin{equation*}
		\int\limits_{\Omega} \Phi(\lambda|f(x)|) \, d\mu(x) < \infty,
	\end{equation*}
	for some $\lambda>0$, endowed with the \textit{Luxemburg norm}
	\begin{equation*}
		\|f\|_{L_{\Phi}(\mu)} := \inf \left\{\varepsilon>0: \int\limits_{\Omega} \Phi\left(\dfrac{|f(x)|}{\varepsilon}\right) \, d\mu(x) \leq 1\right\}.
	\end{equation*}
	A useful and simple observation is that if $f\in L_{\Phi}(\mu)$ satisfies $\int\limits_{\Omega} \Phi\left(|f(x)| / \lambda\right) \, d\mu(x) \geq 1$, then $\lambda\leq \|f\|_{L_{\Phi}(\mu)}$.
	
	Consider now $n \in \mathbb{N}$ and let $\nu = (\nu_{j})_{j=1}^{n} \in S_{\ell_{1}^{n}}$ be a positive sequence. We will denote by $\ell_{\Phi}^{n}(\nu)$ the $n$-dimensional Orlicz sequence space on $(\{1,\dots,n\},2^{\{1,\dots,n\}},\mu_{n})$, where
	\begin{equation*}
		\mu_{n}(\{j\}):= \nu_{j}, \hbox{ for all $j \in \{1,\dots,n\}$},
	\end{equation*}
	endowed with the norm
	\begin{equation*}
		\|(x_{j})_{j=1}^{n}\|_{\ell_{\Phi}^{n}(\nu)} := \inf\left\{\varepsilon > 0: \sum\limits_{j=1}^{n} \Phi\left(\frac{|x_{j}|}{\varepsilon}\right) \nu_{j} \leq 1 \right\}.
	\end{equation*}  
	
	Let $m \in \mathbb{N}$, $E_{1},\dots,E_{m}$ and $F$ be Banach spaces and $\Phi: [0,\infty) \to [0,\infty)$ be an Orlicz function. According to Mastylo and Sanchez \cite{mastylo}, a continuous $m$-linear operator $u: E_{1} \times \dots \times E_{m} \to F$ is \textit{$\Phi$-summing} if there is a constant $C>0$ such that, for all $n \in \mathbb{N}$, for all $\nu = (\nu_{j})_{j=1}^{n} \in S_{\ell_{1}^{n}}$ and $(x_{1,j},\dots,x_{m,j}) \in E_{1} \times \cdots \times E_{m}$, $j \in \{1,\dots,n\}$, we have
	\begin{equation}\label{Phi}
		\begin{aligned}
			& \|(\|u(x_{1,j},\dots,x_{m,j})\|_{F})_{j=1}^{n}\|_{\ell_{\Phi}^{n}(\nu)} \\
			& \leq C \sup\limits_{(\alpha_{1},\dots,\alpha_{m}) \in B_{E_{1}^{*}} \times \dots \times B_{E_{m}^{*}}} \|(\alpha_{1}(x_{1,j}) \dots \alpha_{m}(x_{m,j}))_{j=1}^{n}\|_{\ell_{\Phi}^{n}(\nu)}.    
		\end{aligned} 
	\end{equation}
	The infimum of all constants $C>0$ satisfying \eqref{Phi} is denoted by $\pi_{\Phi}(u)$.
	
	\begin{theorem}\label{domination thm 1}$($\cite[Theorem 3.1]{mastylo}$)$
		Let $m \in \mathbb{N}$, $K_{1},\dots,K_{m}$ be compact Hausdorff topological spaces, $H_{i}$ be a closed subspace of $C(K_{i})$, $i \in \{1,\dots,m\}$, $F$ be a Banach space and $\Phi: [0,\infty) \to [0,\infty)$ be a normalized Orlicz function. If $u: H_{1} \times \dots \times H_{m} \to F$ is a $\Phi$-summing operator, then there are a constant $C>0$ and a regular Borel probability measure $\nu$ on $K_{1} \times \dots \times K_{m}$ endowed with the product topology such that, for all $(T_{1},\dots,T_{m}) \in H_{1} \times \dots \times H_{m}$, we have
		\begin{equation}\label{varphi-summing}
			\|u(T_{1},\dots,T_{m})\|_{F} \leq C \|\odot(T_{1},\dots,T_{m})\|_{L_{\Phi}(\nu)},
		\end{equation}
		where $\odot(T_{1},\dots,T_{m}): K_{1} \times \dots \times K_{m} \to \mathbb{R}$ is the continuous map given by
		\begin{equation*}
			\odot(T_{1},\dots,T_{m})(\varphi_{1},\dots,\varphi_{m}):= T_{1}(\varphi_{1})\dots T_{m}(\varphi_{m}), \hbox{ for all $(\varphi_{1},\dots,\varphi_{m}) \in K_{1} \times \dots \times K_{m}$}.
		\end{equation*}
	\end{theorem}
		
	The least constant $C$ for which the inequality \eqref{varphi-summing} holds is precisely $\pi_{\Phi}(u)$. A measure $\nu$ on $K_{1} \times \dots \times K_{m}$ that satisfies \eqref{varphi-summing} is also called Pietsch measure for $u$.

	\begin{theorem}\label{resultado para os phi somantes}
		Let $m \in \mathbb{N}$, $G_{1},\dots,G_{m}$ be compact Hausdorff topological groups, $H_{i}$ be a closed translation invariant subspace of $C(G_{i})$, $i \in \{1,\dots,m\}$, $F$ be a Banach space, $\Phi: [0,\infty) \to [0,\infty)$ be a normalized Orlicz function and $u: H_{1} \times \dots \times H_{m} \to F$ be a translation $m$-invariant $\Phi$-summing operator. Then the normalized Haar measure on $G_{1} \times \dots \times G_{m}$ is a Pietsch measure for $u$.
	\end{theorem}
	
	\begin{proof}
		Let $G = G_{1} \times \dots \times G_{m}$. We fix $(T_{1},\dots,T_{m}) \in H_{1} \times \dots \times H_{m}$ such that $\|u(T_{1},\dots,T_{m})\|_{F} \neq 0$ and define $D: G \to [0,\infty)$ by
		\begin{equation}\label{definition D}
			D(\varphi) :=  \Phi\left(\dfrac{\pi_{\Phi}(u)|T_{1}(\varphi_{1})\dots T_{m}(\varphi_{m})|}{\|u(T_{1},\dots,T_{m})\|_{F}}\right), \hbox{ for all $\varphi=(\varphi_{1},\dots,\varphi_{m}) \in G$}.
		\end{equation}
		Clearly $D$ is well defined and is continuous. Since the operator $u$ and $H_{1},...,H_{m} $ are translation invariant, for all $\psi=(\psi_{1},\dots,\psi_{m}) \in G$, we have
		\begin{equation}\label{conta 1}
			\|u(T_{1} \circ I_{\psi_{1}}, \dots, T_{m} \circ I_{\psi_{m}})\|_{F} = \|u(T_{1},\dots,T_{m})\|_{F} \neq 0.
		\end{equation}
		By Theorem \ref{domination thm 1}, there is a regular Borel probability measure $\nu$ on $G$ such that, for all $\psi=(\psi_{1},\dots,\psi_{m}) \in G$,
		\begin{equation*}
			1  \leq \int\limits_{G}  \Phi\left(\dfrac{\pi_{\Phi}(u)|T_{1} \circ I_{\psi_{1}}(\varphi_{1})\dots T_{m} \circ I_{\psi_{m}}(\varphi_{m})|}{\|u(T_{1} \circ I_{\psi_{1}},\dots,T_{m} \circ I_{\psi_{m}})\|_{F}}\right) \, d\nu(\varphi).
		\end{equation*}

		It follows from \eqref{definition D} and \eqref{conta 1} that
		\begin{equation*}
			\begin{aligned}
				1 & \leq \int\limits_{G}  \Phi\left(\dfrac{\pi_{\Phi}(u)|T_{1}(\psi_{1}\varphi_{1})\dots T_{m}(\psi_{m}\varphi_{m})|}{\|u(T_{1},\dots,T_{m})\|_{F}}\right) \, d\nu(\varphi) = \int\limits_{G} D(\psi\varphi) \, d\nu(\varphi).
			\end{aligned}
		\end{equation*}
		Thus $\nu$ is $D$-dominating. By Theorem \ref{main theorem}, $\sigma_{G}$ is $D$-dominating. Hence, for all $\psi=(\psi_{1},\dots,\psi_{m}) \in G$, we have
		\begin{equation*}
			1 \leq \int\limits_{G} D(\psi\varphi) \, d\sigma_{G}(\varphi) = \int\limits_{G} D(\varphi) \, d\sigma_{G}(\varphi) = \int\limits_{G}  \Phi\left(\dfrac{\pi_{\Phi}(u)|T_{1}(\varphi_{1})\dots T_{m}(\varphi_{m})|}{\|u(T_{1},\dots,T_{m})\|_{F}}\right) \, d\sigma_{G}(\varphi).
		\end{equation*}
		This implies
		\begin{equation*}\label{conta Phi}
			\|u(T_{1},\dots,T_{m})\|_{F} \leq \pi_{\Phi}(u) \|\odot(T_{1},\dots,T_{m})\|_{L_{\Phi}(\sigma_{G})}.
		\end{equation*}
	Consequently, $\sigma_{G}$ is a Pietsch measure for $u$.
	\end{proof}
	
	In this fashion, Theorem \ref{resultado para os phi somantes} recovers the corresponding result for this class of mappings \cite[Corollary 5.2]{mastylo}.


\begin{thebibliography}{99}  
		
		%\bibitem{alencar matos} R. Alencar and M.C. Matos, \textit{Some classes of multilinear mappings between Banach spaces}, Publ. Dep. Anál. Mat. Univ. Complutense Madrid, Sect. 1, no. 12, 1989.
		
		%\bibitem{ash} R. Ash, \textit{Measure, Integration and Functional Analysis}, Academic Press, Inc., 1972.
		
		%\bibitem{bartle} R. Bartle, \textit{The Elements of Integration and Lebesgue Measure}, Eastern Michigan University and University of Illinois, John Wiley \& Sons, Inc, New York', 1995.
		
		\bibitem{bogachev v1} V. I. Bogachev, \textit{Measure Theory Volume 1}, Springer-Verlag, 2007. 
		
		\bibitem{bogachev v2} V. I. Bogachev, \textit{Measure Theory Volume 2}, Springer-Verlag, 2007. 
		
		\bibitem{BC} G. Botelho and J. R. Campos, 
		\textit{On the transformation of vector-valued sequences by linear and multilinear operators},
		Monatsh. Math. \textbf{183} (2017), no. 3, 415--435.
		
		\bibitem {BMPS}G. Botelho, M. Maia, D. Pellegrino, and J. Santos, \textit{A
			unified factorization theorem for Lipschitz summing operators}, Q. J. Math.,
		\textbf{70} (2019), no. 4, 1521--1533.
		
		\bibitem{unified PDT} G. Botelho, D. Pellegrino and P. Rueda, \textit{A unified Pietsch domination theorem}, J. Math. Anal. Appl. \textbf{365} (2010), 269--276.
		
		\bibitem{haar botelho} G. Botelho, D. Pellegrino, P. Rueda, J. Santos and J. B. Seoane-Sepúlveda, \textit{When is the Haar measure a Pietsch measure for nonlinear mappings?}, Studia Mathematica \textbf{213} (2012), 275--287.
		
		%\bibitem{livro botelho} G. Botelho, D. Pellegrino, E. Teixeira, \textit{Fundamentos de Análise Funcional}, Sociedade Brasileira de Matemática, 2015.
		
		\bibitem {CD}E. \c{C}al\i\c{s}kan and D. Pellegrino, \textit{On the
			multilinear generalizations of the concept of absolutely summing operators},
		Rocky Mountain J. Math., \textbf{37} (2007), 1137--1154.
		
		\bibitem{ca} {\small H. Cartan, \textit{Sur la mesure de Haar}, C.
			R. Acad. Sci. Paris \textbf{211} (1940), 759--762.%
		}
		
		\bibitem{ChenZeheng}  D. Chen and B. Zheng, \textit{Lipschitz $p$-integral operators and Lipschitz $p$-nuclear operators}, Nonlinear Anal. \textbf{75} (2012), 5270--5282. 
		
		%		\bibitem{CalPel} E. \c Caliskan and  D. Pellegrino, \textit{On the multilinear generalizations of the concept of absolutely summing operators}, Rocky Mountain Journal of Mathematics \textbf{37} (2007), 1137–1154.
		
		\bibitem{diestel} J. Diestel, H. Jarchow and A. Tonge, \textit{Absolutely Summing Operators}, Cambridge Stud. Adv. Math. 43, Cambridge Univ. Press., Cambridge, 1995.
		
		\bibitem{Far.Joh.} J. Farmer and W.B. Johnson, \textit{Lipschitz $p$-summing operators}, Proc. Amer. Math. Soc. \textbf{137} (2009), 2989–2995.
		
		\bibitem{folland} G. Folland, \textit{Real Analysis: Modern Techniques and Their Applications}, 2nd Edition, Wiley, 1999.
		
		\bibitem{A.Haar} A. Haar, \textit{Der Massbegriff in der Theorie der kontinuierlichen Gruppen}, Ann. of Mathematics \textbf{34} (1933), 147--169.
		
		\bibitem{LSS} A. C. B. Lima, J. Santos and E. B. Silva, \textit{Splitting Property for Summing Operators}, Bull. Braz. Math. Soc. \textbf{57} (2026), no. 17, 1--16
		
		\bibitem{israel}R. Macedo, D. Pellegrino and J. Santos, \textit{Nonlinear variants of a theorem of Kwapie\'{n}},
		Israel J. Math. \textbf{247} (2022), no. 1, 217--231.
		
		\bibitem{mastylo} M. Mastylo and E. A. Sánchez Pérez, \textit{Ideals of multilinear mappings via Orlicz spaces and translation invariant operators}, Math. Nachr. \textbf{294} (2020), 1--21.
		
		\bibitem{mastyloperez} M. Mastylo and E. A. Sánchez Pérez, \textit{Lipschitz $(q,p)$-summing maps from $C(K)$-spaces to metric spaces}, J. Geom. Anal.  \textbf{33} (2023), no. 4, Paper No. 113, 24 pp.
		
		\bibitem{Neu} {\small J. von Neumann, \textit{Zum Haarschen mass in
				topologischen Gruppen}, Compositio Math. \textbf{1} (1935), 106--114. }
		
		\bibitem{general PDT} D. Pellegrino and J. Santos, \textit{A general Pietsch domination theorem}, J. Math. Anal. Appl. 375 (2011), 371--374.
		
		\bibitem{tunisia} D. Pellegrino and J. Santos, \textit{On summability of nonlinear operators},
		Tunis. J. Math. \textbf{6} (2024), no. 1, 115--135.
		
		\bibitem {adv} D. Pellegrino, J. Santos and J. B. Seoane-Sep\'{u}lveda, \textit{Some techniques on nonlinear analysis and applications}, Advances in Mathematics \textbf{229} (2012), 1235--1265.
		
		\bibitem {stu} A. Pietsch, \textit {Absolut $p$-summierende Abbildungen in normieten R\"{a}umen}, Studia Mathematica \textbf{27} (1967), 333--353.
		
		%\bibitem{Liv.Lip.}  N. Weaver, \textit{Lipschitz Algebras}, 2nd Edition, World Scientific, 2018.
		
		\bibitem{weil} A. Weil, \textit{L'intégration dans les groupes topologiques et ses applications}, 10th Edition, Actualités Scientifiques et Industrielles, 1940.
	\end{thebibliography}
\end{document}